\documentclass[12pt,reqno]{amsart}
\usepackage{hyperref}
\usepackage{amsmath,amssymb}
\usepackage[english]{babel}
\usepackage{caption}
\usepackage{amssymb,amsmath,euscript,enumerate,tikz}
\usepackage[margin=1in]{geometry}
\usepackage{graphicx}
\usepackage{float}
\usepackage{pgf,tikz}
\usepackage{mathrsfs}
\usepackage{upgreek}
\usepackage{mathtools}

\newcommand \reg{\operatorname{reg}}
\newcommand\oc{\operatorname{oc}}
\newcommand \Tor{\operatorname{Tor}}

\newcommand \pd{\operatorname{pd}}

\newcommand \K{\mathbb{K}}

\newcommand{\KK}{\mathbb{K}}

\renewcommand{\vec}[1]{\mathbf{#1}}
\newtheorem{theorem}{Theorem}[section]

\newtheorem{lemma}[theorem]{Lemma}
\newtheorem{proposition}[theorem]{Proposition}

\newtheorem{remark}[theorem]{Remark}
\newtheorem{corollary}[theorem]{Corollary}

\begin{document}
	\title[Parity binomial edge ideals]{Regularity of parity binomial edge ideals}
	\author[Arvind Kumar]{Arvind Kumar}
	\email{arvkumar11@gmail.com}
	\address{Department of Mathematics, Indian Institute of Technology
		Madras, Chennai, INDIA - 600036}

	\begin{abstract}
		Let $G$ be a  simple graph on $n$ vertices and $\mathcal{I}_G$ denotes  parity binomial edge ideal of $G$ in the polynomial ring $S = \K[x_1,
		\ldots, x_n, y_1, \ldots, y_n].$ We obtain lower bound for the regularity of parity binomial edge ideals of graphs.  We then  classify all graphs whose parity binomial edge ideals have regularity $3$.  We   classify graphs whose parity binomial edge 
		ideals have pure resolution.  
	\end{abstract}
	\maketitle
	\section{Introduction}
	The parity binomial edge ideal of a graph was introduced by Kahle et al. in \cite{KCT}.
	Let $\K$ be any field. Let $G$ be a  simple graph on the vertex set $V(G) =[n]:= \{1, \ldots, n\}$. Let $S= \K[x_1, \ldots, x_{n}, y_1, \ldots, y_{n}]$ be the polynomial ring in $2n$ variables with $\deg(x_i) =\deg(y_i)=1$.  The parity binomial edge ideal of a graph $G$ is defined as $$
	\mathcal{I}_G=(x_ix_j-y_iy_j : \{i,j\} \in E(G)) \subset S.$$ Many algebraic properties such as primary decomposition, mesoprimary decomposition, Markov bases, and radicality of parity binomial edge ideal were studied in \cite{KCT}.  Bolognini et al. \cite[Corollary 6.2]{dav} proved  that if $G$ is a bipartite graph, then the parity binomial edge ideal of $G$ is essentially the same as the binomial edge ideal of  $G$. The binomial edge ideal of a graph $G$ was introduced by Herzog et al. in \cite{HH1} and independently by
	Ohtani in \cite{oh}, and it is  defined  as $$J_G = (
	x_i y_j - x_j y_i ~ : i < j, \; \{i,j\} \in E(G)) \subset
	S.$$  In \cite{dav}, Bolognini et al. studied the Cohen-Macaulayness of binomial edge ideals of bipartite graphs. The algebraic properties of binomial edge ideals of Cohen-Macaulay bipartite graphs such as regularity, extremal Betti number were studied in \cite{JACM, CG19}. In \cite{AR2}, we characterized all the graphs whose parity binomial edge ideals are complete intersections or almost complete intersections. We obtained the second graded Betti number and first Syzygy of parity binomial edge ideals of trees and unicyclic graphs in \cite{AR2}. In this article, we aim to study the regularity of parity binomial edge ideals. One of the reasons to explore the regularity of parity binomial edge ideals of graphs is Kahle and  Kr\"usemann conjecture. In \cite{KK19}, Kahle and Kr\"usemann conjectured that if $G$ is a connected graph on $[n]$, then $\ell(G) \leq \reg(S/\mathcal{I}_G) \leq n$, where $\ell(G)$ is the length of a longest induced path in $G$. This conjecture is trivially true for the class of bipartite graphs due to Matsuda and Murai, \cite{MM}. They proved that if $G$ is a connected graph on $[n]$, then  $\ell(G) \leq \reg(S/J_G) \leq n-1$. Jayanthan et al. obtained improved bounds for the regularity of binomial edge ideals of trees, \cite{JNR, JNR1}.  We aim to prove this conjecture for some class of graphs. Another reason to study
	the regularity for parity binomial edge ideal is that the precise expression for the regularity of parity binomial edge ideal is known
	only for some class of graphs, see \cite{THT, JACM, PZ}. We aim to classify graphs whose parity binomial edge ideals have regularity three.
	
	Another important binomial ideal associated to a graph is Lov\'asz-Saks-Schrijver(LSS) ideal of a graph. The ideal $$L_G= (x_ix_j+y_iy_j : \{i,j\} \in E(G))\subset S$$ is known as LSS ideal of the graph $G$.  We refer the reader to \cite{lss89, lov79} for more information on LSS ideals of graphs. In \cite{her3}, Herzog et al. proved that if char$(\K) \neq 2$, then $L_G$ is a radical ideal. Also, they computed the primary decomposition of $L_G$ when $\sqrt{-1} \notin \K$ and char$(\K)\neq 2$. In \cite{AV19}, Conca and Welker studied the algebraic properties  of $L_G$. They proved that $L_G$ is complete intersection if and only if  $G$ does not contain any claw  or an  even cycle (\cite[Theorem 1.4]{AV19}).  In \cite{AR2}, we  proved that $\beta_{i,j}^S(S/L_G) =\beta_{i,j}^S(S/\mathcal{I}_G)$ for all $i,j$ due to which $\reg(S/L_G)=\reg(S/\mathcal{I}_G)$.

	In \cite{her3}, Herzog et al.  introduced the notation of permanental edge ideals of graphs.	The permanental edge ideal of a graph $G$ is denoted by $\Pi_G$, and it is defined as 
	$$\Pi_G =(x_iy_j+x_jy_i : \{i,j\} \in E(G)) \subset S.$$ Herzog et al.  proved that permanental edge ideal of a graph is a radical ideal, in \cite{her3}. If char($\K) \neq 2$, then the permanental edge ideal of a graph is essentially the same as the parity binomial edge ideal of that graph (see Remark \ref{rmk-parity}) which implies that $\reg(S/\Pi_G)=\reg(S/\mathcal{I}_G)$. Thus, the study of regularity of parity binomial edge ideals become more important.
	
	The article is organized as follows. We collect the notation and related definitions in the second section. In Section $3$, we obtain the regularity of parity binomial edge ideals of cycles. Then, we prove Kahle and Kr\"usemann conjecture for some non-bipartite graphs. We also get a regularity lower bound for parity binomial edge ideals of graphs. We then characterize all graphs whose parity binomial edge ideals have regularity $3$. Also, we describe graphs whose parity binomial edge ideals have the pure resolution. 
	
	\section{Preliminaries}
	In this section, we collect all the notions that we will use in this paper. We first recall the necessary definitions from graph theory.
	
	Let $G$  be a  simple graph with the vertex set $[n]$ and edge set
	$E(G)$. A simple graph on the vertex set $[n]$ is said to be a \textit{complete graph}, if
	$\{i,j\} \in E(G)$ for all $i,j \in [n]$. Complete graph on
	$[n]$ is denoted by $K_n$. For $A \subseteq V(G)$, $G[A]$ denotes the
	\textit{induced subgraph} of $G$ on the vertex set $A$, i.e., for
	$i, j \in A$, $\{i,j\} \in E(G[A])$ if and only if $ \{i,j\} \in
	E(G)$.  For a vertex $v$, $G \setminus v$ denotes the  induced
	subgraph of $G$ on the vertex set $V(G) \setminus \{v\}$.  A vertex $v \in V(G)$ is
	said to be a \textit{cut vertex} if $G \setminus v$ has  more
	connected components than $G$. A subset
	$U$ of $V(G)$ is said to be a \textit{clique} if $G[U]$ is a complete
	graph. A vertex $v$ of $G$ is said to be a \textit{free vertex}
	if $v$ is contained in only one maximal clique. 
	The
	\textit{neighbourhood} of a vertex $v$ in $G$ is defined as $N_G(v) := \{u \in V(G) :  \{u,v\} \in E(G)\}$. For a vertex $v$,  $G_v$ is the graph on the vertex set
	$V(G)$ and edge set $E(G_v) =E(G) \cup \{ \{u,w\}: u,w \in N_G(v)\}$.
	The \textit{degree} of a vertex  $v$, denoted by $\deg_G(v)$, is
	$|N_G(v)|$.   For an edge $e$ in $G$, $G\setminus e$ is the graph
	on the vertex set $V(G)$ and edge set $E(G) \setminus \{e\}$.  Let
	$u,v \in V(G)$ be such that $e=\{u,v\} \notin E(G)$. Then,
	$G_e$ denote a graph on the vertex set $V(G)$ and edge set $E(G_e) = E(G) \cup
	\{\{x,y\} : x,\; y \in N_G(u) \; or \; x,\; y \in N_G(v) \}$.  A
	connected graph $G$ is called a \textit{cycle} if  $\deg_G(v) = 2$  for all $v \in V(G)$. 
	A cycle on $n$ vertices is denoted by $C_n$.  A graph is a \textit{tree} if it does not contain a
	cycle. A
	graph is said to be a \textit{unicyclic} graph if it contains exactly
	one cycle. The \textit{girth} of a graph $G$ is the length of  a shortest
	cycle in $G$. A unicyclic graph with even girth is called an \textit{even unicyclic} and with odd girth is called  an \textit{odd unicyclic} graph.   A graph $G$ is said to be \textit{bipartite} if 
	there is a bipartition of $V(G)=V_1 \sqcup V_2$ such that for each
	$i=1,2$, no two of the vertices of $V_i$ are adjacent. A graph is  called a  \textit{non-bipartite} graph if it is not a bipartite graph.   A maximal subgraph of $G $ without a cut vertex is called a  \textit{block} of $G$. A graph $G$ is said to be a  \textit{block} graph if each  block of $G$ is a clique. A graph on $4$ vertices with $5$ edges is said to be a \textit{diamond graph}. 
	
	Now, we recall the  necessary notation from commutative algebra.
	Let $R = \KK[x_1,\ldots,x_m]$ be a standard graded polynomial ring over an arbitrary
	field $\KK$ and $M$ be a finitely generated graded  $R$-module. 
	Let
	\[
	0 \longrightarrow \bigoplus_{j \in \mathbb{Z}} R(-j)^{\beta_{p,j}^R(M)} 
	\overset{\phi_{p}}{\longrightarrow} \cdots \overset{\phi_1}{\longrightarrow} \bigoplus_{j \in \mathbb{Z}} R(-j)^{\beta_{0,j}^R(M)} 
	\overset{\phi_0}{\longrightarrow} M\longrightarrow 0,
	\]
	be the minimal graded free resolution of $M$, where
	$R(-j)$ is the free $R$-module of rank $1$ generated in degree $j$.
	The number $\beta_{i,j}^R(M)$ is  called the 
	$(i,j)$-th graded Betti number of $M$. The \textit{regularity} of $M$, denoted by $\reg(M)$, is defined as \[\reg(M):=\max\{j-i : \beta_{i,j}^R(M) \neq 0, \; i \geq 0\}.\]
	
	The following fundamental property of   regularity is used repeatedly in this
	article. We refer the reader to the book \cite[Chapter 18]{Peeva} for more properties of regularity.
	
	\begin{lemma}\label{regularity-lemma}
		Let ${M},{N}$ and ${P}$ be finitely generated graded ${S}$-modules. 
		If $$ 0 \rightarrow {M} \xrightarrow{f}  {N} \xrightarrow{g} {P} \rightarrow 0$$ is a 
		short exact sequence with $f,g$  
		graded homomorphisms of degree zero, then 
		\begin{enumerate}
			\item $\reg({P}) \leq \max\{\reg({N}),\reg({M})-1\}$.
			\item $\reg({P}) = \reg({N})$, if $\reg({N}) >\reg({M})$. 
		\end{enumerate}
	\end{lemma} 
	
	Let $G$ be a graph on $[n]$ and $S=\K[x_1,\ldots,x_n,y_1,\ldots,y_n]$. For an edge $e = \{i,j\}\in E(G)$ with $i <
	j$, we define $f_e = f_{i,j} :=x_i y_j - x_j y_i$,  $g_e=g_{i,j}:=x_ix_j+y_iy_j$ and $\bar{g}_e=\bar{g}_{i,j}:=x_ix_j-y_iy_j$.
	\section{Parity Binomial Edge Ideals}
	In this section, we study some homological invariants associated with the minimal free resolution of $S/\mathcal{I}_G$. We  classify graphs whose parity binomial edge ideals have regularity $3$. We characterize  graphs whose parity binomial edge  ideals have the pure resolution.  
	
	First, we recall a fact about bipartite graphs from \cite{dav}. 
	\begin{remark}\cite[Corollary 6.2]{dav}\label{main-rmk}\rm{
			Let $G$ be a bipartite graph with bipartition $[n]=V_1\sqcup V_2$. We define 
			$\varPhi : S \rightarrow S$ as \[\varPhi(x_i)= \left\{
			\begin{array}{ll}
			x_i & \text{ if } i \in V_1 \\
			y_i & \text{ if } i \in V_2
			\end{array} \right.  \text{  and  } \;~  \varPhi(y_i)= \left\{
			\begin{array}{ll}
			y_i & \text{ if } i \in V_1 \\
			x_i & \text{ if } i \in V_2.
			\end{array} \right.\] It is clear that $\varPhi$ is an isomorphism and  $\varPhi(J_G)=\mathcal{I}_G$.  }
	\end{remark}
	We give regularity upper bound for parity binomial edge  ideals of bipartite graphs. 
	\begin{theorem}\label{reg-bipartite}
		Let $G$ be a connected bipartite graph on $[n]$ Then, $\reg(S/\mathcal{I}_G)=\reg(S/J_G) \leq n-1$. Moreover, $\reg(S/\mathcal{I}_G) =n-1$ if and only if $G$ is a path graph.
	\end{theorem}
	\begin{proof}
		The  first assertion immediately follows from Remark \ref{main-rmk} and \cite[Theorem 1.1]{MM}. The second assertion follows from \cite[Theorem 3.2]{KM3}.
	\end{proof}
	
	Here we recall some facts about permanental edge ideals of graphs. 
	Due to the following remark, if char$(\K)\neq 2$, then $\Pi_G$ and $\mathcal{I}_G$ are essentially the same  and if char$(\K) = 2$, then $\Pi_G=J_G$.
	\begin{remark}\label{rmk-parity}\rm{
			Let $G$ be a graph on the vertex set $[n]$. If char$(\K) = 2$, then $\mathcal{I}_G =L_G$ and $\Pi_G =J_G$.  Suppose char($\K)\neq 2$.  We define $\eta : S \rightarrow S$ as 
			\[\eta(x_i)=x_i +  y_i 
			\text{  and  } \;~  \eta(y_i)= 
			x_i- y_i \text{ for all } i \in V(G)
			.\] It is clear that $\eta$ is an isomorphism and $\Pi_G=\eta(\mathcal{I}_G)$.}
	\end{remark}
	
	Now, we move on to study regularity upper bound for parity binomial edge ideals of some non-bipartite graphs. First, we compute the regularity of parity binomial edge ideal of an odd cycle.
	\begin{theorem}\label{reg-odd-cycle}
		Let $G =C_n$, where $n$ is odd. Then, $\reg(S/\mathcal{I}_G) =n$.
	\end{theorem}
	\begin{proof}
		It follows from \cite[Theorem 3.4]{AR2} that $\mathcal{I}_G$ is a complete intersection with $n$ generators of degree two. Therefore, $\reg(S/\mathcal{I}_G) =n$.
	\end{proof}
	To compute the regularity upper bound for  parity binomial edge ideals of some non-bipartite graphs, we need the following lemma from \cite{AR2}.
	
	\begin{lemma}\label{colon-non-bipartite}\cite[Lemma 3.4]{AR2}
		Let $G$ be a non-bipartite graph on $[n]$. Assume that there exists  $e=\{u,v\} \in E(G)$ such that $G\setminus e$ is a bipartite graph. Then, $$\mathcal{I}_{G\setminus e}:\bar{g}_e=\mathcal{I}_{G\setminus e}+(f_{i,j}: i,j \in N_{G\setminus e}(u) \text{ or } i,j \in N_{G\setminus e}(v))=\varPhi(J_{(G\setminus e)_e}).$$
	\end{lemma}
	
	Now, we  give regularity upper bound for parity binomial edge ideals of some non-bipartite graphs. 
	\begin{theorem}\label{reg-non-bipartite}
		Let $G$ be a connected  non-bipartite graph on $[n]$ such that    $G\setminus e$ is a bipartite graph for some $e=\{u,v\} \in E(G)$. If $G$ is not an odd cycle, then $\reg(S/\mathcal{I}_G) \leq n-1$.
	\end{theorem}
	\begin{proof}
		Consider the short exact sequence 
		\begin{equation}\label{ses-LS}
			0 \longrightarrow \frac{S}{\mathcal{I}_{G\setminus e}:\bar{g}_e}(-2) \stackrel{\cdot \bar{g}_e}\longrightarrow \frac{S}{\mathcal{I}_{G\setminus e}} \longrightarrow \frac{S}{\mathcal{I}_G} \longrightarrow 0.
		\end{equation}
		It follows from Lemma \ref{regularity-lemma} and the above short exact sequence  that \[\reg(S/\mathcal{I}_G) \leq \max \{\reg(S/\mathcal{I}_{G\setminus e}), \reg(S/\mathcal{I}_{G\setminus e}:\bar{g}_e)+1\}.\] Since $G$ is not an odd cycle, $G\setminus e$ is not a path graph. Hence, by Theorem \ref{reg-bipartite}, we have $\reg(S/\mathcal{I}_{G\setminus e}) \leq n-2$. Note that $(G\setminus e)_e$ is not a path graph. Therefore, by \cite[Theorem 3.2]{KM3}, we have $\reg(S/J_{(G\setminus e)_e}) \leq n-2$. Now, by Lemma \ref{colon-non-bipartite}, $\reg(S/\mathcal{I}_{G\setminus e}:\bar{g}_e)= \reg(S/J_{(G\setminus e)_e})\leq n-2$. Hence, the desired result follows.
	\end{proof} 
	
	Now, we obtain a regularity lower bound for parity binomial edge ideals.
	\begin{theorem}\label{betti-induced}
		Let $G$ be a graph and $H$ be an induced subgraph of $G$. Then, $\beta_{i,j}^{S_H}(S_H/\mathcal{I}_H) \leq \beta_{i,j}^S(S/\mathcal{I}_G)$, for all $i,j$, where $S_H=\K[x_k,y_k : k \in V(H)]$.
	\end{theorem}
	\begin{proof}
		First, we claim that $\mathcal{I}_H = \mathcal{I}_G \cap S_H$, where $\mathcal{I}_H$ is the parity binomial edge ideal of $H$ in  $S_H$. Since generators of $\mathcal{I}_H$ are contained in $\mathcal{I}_G$, $\mathcal{I}_H \subset \mathcal{I}_G \cap S_H$. Now, let $g \in \mathcal{I}_G \cap S_H$ be any element. Since $g \in \mathcal{I}_G$, we have \[g = \sum_{ e \in E(G)} h_e \bar{g}_e,\] where  $ h_e \in S$.  Now, set $x_i=y_i =0$, for $i \notin V(H)$ in  $g = \sum_{ e \in E(G)} h_e \bar{g}_e$. As $g \in S_H$, the left-hand side of the equation does not change. In right-hand side, if  $e=\{u,v\}  \in E(G) \setminus E(H)$, then either $u \notin V(H) $ or $v \notin V(H)$. Therefore, we get $g = \sum_{ e \in E(H)} h'_e \bar{g}_e$, where $h'_e$ is obtained from $h_e$ by putting $x_i=y_i =0$ for $i \notin V(H)$. Thus, $g \in \mathcal{I}_H$. Hence, $S_H/\mathcal{I}_H$ is $\K$-subalgebra of $S/\mathcal{I}_G$. Consider, $S_H/\mathcal{I}_H \xhookrightarrow{i} S/\mathcal{I}_G \stackrel{\pi}\rightarrow S_H/\mathcal{I}_H$, where  $\pi$ is an  epimorphism induced by setting $x_i=y_i =0$, for $i \notin V(H)$. Note that $\pi \circ i$ is identity on $S_H/\mathcal{I}_H$. Thus, $S_H/\mathcal{I}_H$ is algebra retract of $S/\mathcal{I}_G$. Hence, the result follows from \cite[Corollary 2.5]{ohh2000}.
	\end{proof}
	Let $\oc(G)$ denote the length of a longest induced odd cycle in $G$, and $\ell(G)$ denote the length of a longest induced path in $G$. If $G$ has no induced odd cycle, we assume that $\oc(G)=0$.
	\begin{corollary}\label{lower-bound}
		Let $G$ be a connected graph on $[n]$. Then, $\reg(S/\mathcal{I}_G)\geq \max\{\ell(G),\oc(G)\}.$
	\end{corollary}
	\begin{proof}
		Let $H$ be a longest induced path in $G$ and $S_H=\K[x_j,y_j : j\in V(H)].$ By Theorem \ref{reg-bipartite},  $\reg(S_H/\mathcal{I}_H)=\ell(G)$. Hence, by Theorem \ref{betti-induced}, $\reg(S/\mathcal{I}_G) \geq \reg(S_H/\mathcal{I}_H)=\ell(G)$.  Assume that $G$ is a non-bipartite graph, then $G$ has an induced odd cycle. Let $H'$ be an induced odd cycle such that $\oc(G)=|V(H')|$. Set $S_{H'}=\K[x_j,y_j : j \in V(H')]$. Then, by Theorem \ref{reg-odd-cycle}, $\reg(S_{H'}/\mathcal{I}_{H'})=\oc(G)$. Hence, by Theorem \ref{betti-induced}, $\reg(S/\mathcal{I}_G) \geq \max \{\ell(G),\oc(G)\}$.
	\end{proof}
	We now characterize the graphs whose parity binomial edge ideals have regularity three.
	\begin{theorem}\label{reg-3}
		Let $G$ be a graph on $[n]$ with no isolated vertex. Then, $\reg(S/\mathcal{I}_G)=2$ if and only if either $G=K_2 \sqcup K_2$ or $G$ is a complete bipartite graph other than $K_2$.
	\end{theorem}
	\begin{proof}
		First assume that $\reg(S/\mathcal{I}_G)=2$. If $G$ is non-bipartite, then it follows from Corollary \ref{lower-bound} that $\reg(S/\mathcal{I}_G)\geq 3$. Thus, $G$ is a bipartite graph with bipartition $V_1 \sqcup V_2$. Assume that $G$ is a disconnected graph with connected components $H_1,\ldots, H_c$. Since each component has at least one edge, by Corollary \ref{lower-bound}, $\reg(S_{H_i}/\mathcal{I}_{H_i}) \geq 1$, where $S_{H_i}=\K[x_j,y_j : j \in V(H_i)]$. Therefore, $c=2$ and $\reg(S_{H_i}/\mathcal{I}_{H_i})=1$, for $i=1,2$. If $H_i \neq K_2$, then $\ell(H_i)\geq 2$, which is not possible. Thus, $G=K_2 \sqcup K_2$. Now, if $G$ is a connected  graph  which is not a complete bipartite graph, then $\ell(G)\geq 3$. Therefore, by Theorem \ref{lower-bound}, $\reg(S/\mathcal{I}_G) \geq 3$ which is not possible. Hence, $G$ is a complete bipartite graph.
		
		Conversely, if $G=K_2 \sqcup K_2$, then $\mathcal{I}_G$ is a complete intersection ideal with two generators of degree two. Thus, $\reg(S/\mathcal{I}_G)=2$. Now, if $G$ is a complete bipartite graph, then by Remark \ref{main-rmk} and \cite[Theorem 1.1]{PZ} $\reg(S/\mathcal{I}_G)=2$.
	\end{proof}
	
	Let $I$ be a homogeneous ideal of $S$ such that all generators have degree $d$. If for each $1 \leq i \leq \pd(S/I)$, there exist a unique $d_i >0$ such that $\beta_{i,d_i}^S(S/I) \neq 0$, then we say that $S/I$ has \textit{pure resolution}. We start preparation for characterization of graphs whose parity binomial edge ideals have pure resolution.
	
	%If the minimal free resolution of $S/I$ is 
	%\[
	%0 \longrightarrow  S(-d_p)^{\beta_{p,d_p}^S(S/I)} 
	%\overset{\phi_{p}}{\longrightarrow} \cdots \overset{\phi_2}{\longrightarrow} S(-d_1)^{\beta_{1,d_1}^S(S/I)} \overset{\phi_1}{\longrightarrow} S
	%\overset{\phi_0}{\longrightarrow} S/I\longrightarrow 0,
	%\]
	\begin{theorem}\label{betti-LSS}
		Let $G$ be a graph on $[n]$. Then, we have
		\begin{enumerate}
			\item $\beta_{2,3}(S/\mathcal{I}_G) =0$. In particular, $\beta_{i,i+1}(S/\mathcal{I}_G)=0$, for all $i\neq 1$.
			\item $\beta_{2,4}(S/\mathcal{I}_G) \geq \binom{|E(G)|}{2}$, if $G \neq K_2$.
			\item $\beta_{i,i}(S/\mathcal{I}_{G \setminus e}:\bar{g}_e)=0$, for $i>0$.
			\item $\beta_{3,6}(S/\mathcal{I}_G) \neq 0$, if $G$ is a non-bipartite graph.
			\item If $G$ is an odd unicyclic graph, then $\beta_{i,j}(S/\mathcal{I}_G)=0$, for $j >2i$.
		\end{enumerate} 
	\end{theorem}
	\begin{proof}\par (1)
		Consider, the minimal free resolution of $S/\mathcal{I}_G$, \[\cdots \longrightarrow S^{|E(G)|}(-2) \stackrel{\varphi}\longrightarrow S \longrightarrow S/\mathcal{I}_G \longrightarrow 0 .\] Let $\{e_{\{i,j\}} : \{i,j\} \in E(G) \}$ be a basis of  $S^{|E(G)|}(-2)$. Note that $\varphi(e_{\{i,j\}}) =\bar{g}_{i,j}$, for all $\{i,j\} \in E(G)$. Let $\epsilon_i$ denote the $i$-th standard basis vector of $\mathbb{Z}^n$. If, we set $\deg(x_i)=\deg(y_i)=\epsilon_i$, then $\mathcal{I}_G$ is $\mathbb{Z}^n$-graded. Therefore, $\deg(e_{\{i,j\}})=\deg(\bar{g}_{i,j})=\epsilon_i+\epsilon_j$. Let $M_1$ be the  first syzygy module of $\mathcal{I}_G$.  Let if possible, \[r=\sum_{\{i,j\} \in E(G)} h_{\{i,j\}}e_{\{i,j\}} \in M_1,\] be a relation of degree three in standard grading. Since $\mathcal{I}_G$ is $\mathbb{Z}^n$-graded, $M_1$ is also $\mathbb{Z}^n$-graded. Therefore, $M_1$ is generated by $\mathbb{Z}^n$-graded homogeneous elements. We can assume that $r$ is $\mathbb{Z}^n$-graded homogeneous element with $\deg(r)=\vec{a} \in \mathbb{Z}^n$ and  total degree $|\vec{a}|=3$, where $|\vec{a}|$ is sum of components of $\vec{a}$. If $h_{\{i,j\}} e_{\{i,j\}} \neq 0$, then $\deg(r)=\vec{a}=\deg(h_{\{i,j\}})+\epsilon_i+\epsilon_j$. Therefore, in standard grading,  degree of $h_{\{i,j\}}$  is one, and hence, $\deg(h_{\{i,j\}}) = \epsilon_k$, for some $k$. If $k=i$ or $k=j$, then $r=h_{\{i,j\}}e_{\{i,j\}}$, which is not possible. If $k\neq i$ and $k \neq j$, then $r=h_{\{i,j\}}e_{\{i,j\}} +h_{\{i,k\}}e_{\{i,k\}}+h_{\{j,k\}}e_{\{j,k\}}$. Thus, $r$ is relation of ideal $(\bar{g}_{i,j},\bar{g}_{i,k},\bar{g}_{j,k})$. It follows from \cite[Theorem 3.4]{AR2} that $(\bar{g}_{i,j},\bar{g}_{i,k},\bar{g}_{j,k})$ is complete intersection. Consequently, there has no relation of degree three in  the first syzygy of $(\bar{g}_{i,j},\bar{g}_{i,k},\bar{g}_{j,k})$, which is contradiction. Hence, the desired result follows. 
		\par (2) Since $G \neq K_2$, $G$ has atleast two edges, say $\{i,j\}, \{k,l\}$. Therefore, $r=\bar{g}_{k,l}e_{\{i,j\}} -\bar{g}_{i,j}e_{\{k,l\}} \in M_1$ is relation of degree four. Since $\beta_{2,3}(S/\mathcal{I}_G)=0$, $r$ is a minimal relation. Hence, $\beta_{2,4}(S/\mathcal{I}_G) \geq \binom{|E(G)|}{2}$.
		\par (3) Consider the  long exact sequence of Tor corresponding to the short exact sequence \eqref{ses-LS}, \[\cdots	\rightarrow \Tor_{2,3}^S\left(\frac{S}{\mathcal{I}_{G}},\K\right) 
		\rightarrow \Tor_{1,3}^S\left(\frac{S}{\mathcal{I}_{G\setminus e}:\bar{g}_e}(-2),\K\right)\rightarrow  \Tor_{1,3}^S\left(\frac{S}{\mathcal{I}_{G\setminus e}},\K\right)\rightarrow \cdots.\] Since $\beta_{1,3}(S/\mathcal{I}_{G\setminus e})=0$,  by part $(1)$, $\beta_{2,3}(S/\mathcal{I}_G)=0$, and  \[\Tor_{1,3}^S\left(\frac{S}{\mathcal{I}_{G\setminus e}:\bar{g}_e}(-2),\K\right) \simeq \Tor_{1,1}^S\left(\frac{S}{\mathcal{I}_{G\setminus e}:\bar{g}_e},\K\right),\] $\beta_{1,1}(S/\mathcal{I}_{G\setminus e}:\bar{g}_e)=0$. Hence, the desired result follows.
		\par (4) Since $G$ is a non-bipartite graph, $G$ has an induced odd cycle, say $H=C_k$. It follows from \cite[Theorem 3.4]{AR2} that $\mathcal{I}_{H}$ is a complete intersection ideal. Therefore, $\beta_{3,6}^{S_H}(S_H/\mathcal{I}_{H}) = \binom{k}{3} \neq 0$, and hence, by Theorem \ref{betti-induced}, $\beta_{3,6}(S/\mathcal{I}_G) \neq 0$.
	\end{proof}
	To prove $(5)$ we need the following theorem.
	\begin{theorem}\label{betti-chordal}
		Let $G$ be a connected chordal graph on $[n]$. Then, $\beta_{i,j}(S/J_G) =0$, for $j>2i$.
	\end{theorem}
	\begin{proof}
		We make  an induction on the number of vertices. If $n=2$, then $G=K_2$, and hence, the result follows. Assume that $n>2$ and the assertion is true for any connected chordal graph $H$ with $|V(H)| \leq n-1$. Suppose that $G$ has a free vertex, say $v$, such that $\deg_G(v)=1$. Let $u$ be the only neighbour of $v$, and $e=\{v,u\}$ be the edge joining $v$ and $u$. Consider, the short exact sequence \begin{align}\label{main-ses}
			0\longrightarrow \frac{S}{J_{G\setminus e}:f_e}(-2) \stackrel{\cdot f_e}{\longrightarrow} \frac{S}{J_{G\setminus e}} \longrightarrow \frac{S}{J_G} \longrightarrow 0 .
		\end{align} It follows from \cite[Theorem 3.4]{FM} that $J_{G\setminus e}:f_e=J_{(G\setminus e)_e}$.  Observe that $J_{G\setminus e} =J_{G\setminus v}$ and $J_{(G\setminus e)_e}=J_{(G\setminus v)_u}$. Now, consider the long exact sequence of Tor corresponding to the short exact sequence \eqref{main-ses}, \begin{equation}\label{Tor-bino}
			\cdots	\rightarrow \Tor_{i,j}^S\left(\frac{S}{J_{G\setminus e}},\K\right)\rightarrow \Tor_{i,j}^S\left(\frac{S}{J_{G}},\K\right) 
			\rightarrow \Tor_{i-1,j}^S\left(\frac{S}{J_{G\setminus e}:f_e}(-2),\K\right)\rightarrow\cdots.
		\end{equation} 
		Notice that \[\Tor_{i-1,j}^S\left(\frac{S}{J_{G\setminus e}:f_e}(-2),\K\right) \simeq \Tor_{i-1,j-2}^S\left(\frac{S}{J_{G\setminus e}:f_e},\K\right)=\Tor_{i-1,j-2}^S\left(\frac{S}{J_{(G\setminus e)_e}},\K\right).\] Since $v$ is an isolated vertex of $G\setminus e$ and $(G\setminus e)_e$, by induction hypothesis, we have $\beta_{i,j}(S/J_{G\setminus e})=0$ and $\beta_{i-1,j-2}(S/J_{(G\setminus e)_e}) =0$, for $j > 2i$. Hence, the long exact sequence \eqref{Tor-bino} yields that  $\beta_{i,j}(S/J_G) =0,$ for $j >2i$.  
		
		Suppose that all the free vertices of $G$ have degree greater than one. Let $v$ be a free vertex of $G$ and $v_1,\ldots,v_t$ be all the neighbors of $v$, and $e_1,\ldots,e_t$ be the edges joining $v$ to $v_1,\ldots,v_t$, respectively, where $t\geq 2$.  
		For each $1 \leq k \leq t-1$, consider the following short exact sequences,
		\begin{align}\label{main-ses1}
			0\longrightarrow \frac{S}{J_{G\setminus \{e_{k+1},\ldots,e_t\}}:f_{e_{k+1}}}(-2) \stackrel{\cdot f_{e_{k+1}}}{\longrightarrow} \frac{S}{J_{G\setminus \{e_{k+1},\ldots,e_t\}}} \longrightarrow \frac{S}{J_{G \setminus \{e_{k+2},\ldots, e_t\}}} \longrightarrow 0 .
		\end{align}
		Set $H_k= (G \setminus \{v_1,\ldots, v_{k}, e_{k+1}, \ldots, e_t\})_{v_{k+1}}$ and $S_{k}= \K[x_j,y_j : j \in V(H_k)]$, for $1 \leq k \leq t-1$. Here, if $k=t-1$, then we assume that $\{e_{k+2},\ldots, e_t\} =\emptyset$. It follows from \cite[Theorem 3.4]{FM} that \[J_{G\setminus \{e_{k+1},\ldots,e_t\}} :f_{e_{k+1}} =(x_{v_i},y_{v_i}: 1 \leq i \leq k)+J_{H_k}.\] 
		Note that for $1 \leq k \leq t-1$, $|V(H_k)| \leq n-1$. By induction hypothesis,  $\beta_{i,j}(S_k/J_{H_k})=0$, for $j >2i$. Since $\{x_{v_i},y_{v_i}: 1 \leq i \leq k\}$ and generators of $J_{H_k}$ are in  disjoint variables, we have \[\beta_{i,j}(S/J_{G\setminus \{e_{k+1},\ldots,e_t\}}:f_{e_{k+1}})=\sum_{l=0}^i \binom{2k}{l} \beta_{i-l,j-l}(S_k/J_{H_k}).\] If $j>2i$, then for each $0 \leq l \leq i$, $j-l>2(i-l)$. Thus,  $ \beta_{i,j}(S/J_{G\setminus \{e_{k+1},\ldots,e_t\}}:f_{e_{k+1}})=0$, for $j >2i$. The long exact sequence of Tor corresponding to the short exact sequence \eqref{main-ses1} with $j >2i$ is 
		\begin{align*}
			\cdots	\rightarrow \Tor_{i,j}^S\left(\frac{S}{J_{G\setminus \{e_{k+1},\ldots,e_t\}}},\K\right)\rightarrow \Tor_{i,j}^S\left(\frac{S}{J_{G\setminus \{e_{k+2},\ldots,e_t\}}},\K\right) 
			\rightarrow 0. 
		\end{align*}
		Therefore, for $j >2i$, $\beta_{i,j}(S/J_{G \setminus \{e_{k+2},\ldots,e_t\}})=0$, if $\beta_{i,j}(S/J_{G\setminus \{e_{k+1},\ldots,e_t\}})=0$. Notice that $v$ is free vertex of degree one in $G\setminus \{e_{2},\ldots,e_t\}$. Consequently, $\beta_{i,j}(S/J_{G\setminus \{e_2,\ldots,e_t\}})=0$, for $j>2i$.  Hence, for $j >2i$, $\beta_{i,j}(S/J_G)=0$.
	\end{proof}
	\textbf{Proof of Theorem \ref{betti-LSS}(5):}  Let $e=\{u,v\}$ be an edge of the cycle. Now, consider the short exact sequence \eqref{ses-LS}.   The long exact sequence of Tor corresponding to \eqref{ses-LS} is 
	\begin{equation}\label{Tor-les}
		\cdots	\rightarrow \Tor_{i,j}^S\left(\frac{S}{\mathcal{I}_{G\setminus e}},\K\right)\rightarrow \Tor_{i,j}^S\left(\frac{S}{\mathcal{I}_{G}},\K\right) 
		\rightarrow \Tor_{i-1,j}^S\left(\frac{S}{\mathcal{I}_{G\setminus e}:\bar{g}_e}(-2),\K\right)\rightarrow\cdots
	\end{equation} Observe that $G \setminus e$ is a tree and  $(G\setminus e)_e$ is a block graph on $[n]$. Since $G\setminus e$ is a bipartite graph, by Remark \ref{main-rmk} and Theorem \ref{betti-chordal}, $\beta_{i,j}(S/\mathcal{I}_{G\setminus e})=0$ for $j>2i$.
	It follows from Lemma \ref{colon-non-bipartite} that $\beta_{i-1,j-2}(S/\mathcal{I}_{G \setminus e}:\bar{g}_e)=\beta_{i-1,j-2}(S/J_{(G\setminus e)_e})$. Since  $(G\setminus e)_e$ is a chordal graph, by Theorem \ref{betti-chordal},  $\beta_{i-1,j-2}(S/\mathcal{I}_{G\setminus e}:\bar{g}_e)=0$ for $j>2i$.  Hence, $\beta_{i,j}(S/\mathcal{I}_G)=0$, for $j>2i$.

	\begin{proposition}\label{diamond-complete}
		Let $G$ be a graph on $4$ vertices. Then, we have
		\begin{enumerate}
			\item If $G$ is a diamond graph, then $S/\mathcal{I}_G$ has not pure resolution.
			\item If $G=K_4$, then $S/\mathcal{I}_{G}$ has not pure resolution.
			\item If $G$ is obtained by attaching an edge to a vertex of $K_3$, then $S/\mathcal{I}_G$ has not pure resolution.
		\end{enumerate}
	\end{proposition}
	\begin{proof}
		Note that in each case $G$ is a non-bipartite graph. Therefore, by Theorem \ref{betti-LSS}, $\beta_{3,6}(S/\mathcal{I}_G) \neq 0$. Now, we prove that $\beta_{3,5}(S/\mathcal{I}_G) \neq 0$ in each case.	
		\par (1)   There is an edge $e=\{u,v\}$ such that $G\setminus e=C_4$. Consider the long exact sequence of Tor corresponding to short exact sequence \eqref{ses-LS}, \[
		\cdots	\rightarrow \Tor_{3,5}^S\left(\frac{S}{\mathcal{I}_{G}},\K\right) 
		\rightarrow \Tor_{2,5}^S\left(\frac{S}{\mathcal{I}_{G\setminus e}:\bar{g}_e}(-2),\K\right)\rightarrow \Tor_{2,5}^S\left(\frac{S}{\mathcal{I}_{G\setminus e}},\K\right)\rightarrow \cdots.\]
		Note that $G\setminus e$ is complete bipartite graph. By Theorem \ref{reg-3}, $\reg(S/\mathcal{I}_{G\setminus e})=2$, therefore, $\beta_{2,5}(S/\mathcal{I}_{G\setminus e})=0$. Notice that \[\Tor_{2,5}^S\left(\frac{S}{\mathcal{I}_{G\setminus e}:\bar{g}_e}(-2),\K\right) \simeq \Tor_{2,3}^S\left(\frac{S}{\mathcal{I}_{G\setminus e}:\bar{g}_e},\K\right).\] By Lemma \ref{colon-non-bipartite}, $\beta_{2,3}(S/(\mathcal{I}_{G\setminus e}:\bar{g}_e))=\beta_{2,3}(S/J_{(G\setminus e)_e})$. Since $(G\setminus e)_e$ is a diamond graph, by \cite[Theorem 2.2]{KM12}, $\beta_{2,3}(S/\mathcal{I}_{G\setminus e}:\bar{g}_e)=4$. The above long exact sequence yields that  $\beta_{3,5}(S/\mathcal{I}_G) \neq 0$. 
		\par (2) Let $e \in E(G)$ be an edge. Then, $G\setminus e$ is a diamond graph. It follows from Theorem \ref{betti-LSS} that $\beta_{3,3}(S/\mathcal{I}_{G\setminus e}:\bar{g}_e)=0$. Since \[\Tor_{3,5}^S\left(\frac{S}{\mathcal{I}_{G\setminus e}:\bar{g}_e}(-2),\K\right) \simeq \Tor_{3,3}^S\left(\frac{S}{\mathcal{I}_{G\setminus e}:\bar{g}_e},\K\right),\] the long exact sequence of Tor corresponding to short exact sequence \eqref{ses-LS},
		\[
		\cdots 0\rightarrow \Tor_{3,5}^S\left(\frac{S}{\mathcal{I}_{G\setminus e}},\K\right)\rightarrow \Tor_{3,5}^S\left(\frac{S}{\mathcal{I}_{G}},\K\right) \rightarrow   \cdots.\] Now, if $\beta_{3,5}(S/\mathcal{I}_G)=0$, then $\beta_{3,5}(S/\mathcal{I}_{G\setminus e})=0$, which is contradiction to part $(1)$. Therefore, $\beta_{3,5}(S/\mathcal{I}_G) \neq 0$.
		\par (3) Let $e \in E(G)$ such that $G\setminus e$ is a claw. It follows from \cite[Theorem 5.3]{PZ} that $\beta_{3,5}(S/\mathcal{I}_{G\setminus e}) \neq 0$. By Theorem \ref{betti-LSS}, $\beta_{3,3}(S/\mathcal{I}_{G\setminus e}:\bar{g}_e)=0$. Hence, $\beta_{3,5}(S/\mathcal{I}_G) \neq 0$.
	\end{proof}
	We now move on to characterize graphs whose parity binomial edge ideals have the pure resolution. 
	\begin{lemma}\label{induced-pure}
		Let $G$ be a graph on $[n]$ and $H$ be an induced subgraph of $G$. If $S/\mathcal{I}_G$ has the pure resolution, then $S_H/\mathcal{I}_H$ has pure resolution.
	\end{lemma}
	\begin{proof}
		The assertion immediately follows from Theorem \ref{betti-induced}. 
	\end{proof} 
	\begin{theorem}
		Let $G$ be a  graph on $[n]$. Then, $S/\mathcal{I}_G$ has  pure resolution if and only if $G$ is one of the following: 
		\begin{enumerate}
			\item $G$ is a complete bipartite graph.
			\item $G$ is a disjoint union of some  odd cycles and some paths.
		\end{enumerate}
	\end{theorem}
	\begin{proof}
		If $G$ is a complete bipartite graph, then by \cite[Theorem 5.3]{PZ} and by  Remark \ref{main-rmk}, $S/\mathcal{I}_G$ has  pure resolution. If $G$ is disjoint union of some odd cycles  and some paths, then it follows from \cite[Corollary 3.6]{AR2} that $\mathcal{I}_G$ is  a complete intersection. Therefore, the Koszul  complex is a minimal free resolution of $S/\mathcal{I}_G$. Hence, $S/\mathcal{I}_G$ has  pure resolution.
		
		Conversely,  we assume that $G$ is a bipartite graph and $S/\mathcal{I}_G$ has  pure resolution. Then, it follows from  Remark \ref{main-rmk} and \cite[Theorem 2.2]{KM4} that $G$ is either a complete bipartite graph or a disjoint union of some path graphs. So, now we assume that $G$ is a connected non-bipartite graph such that $S/\mathcal{I}_G$ has  pure resolution. We claim that $G$ is an odd cycle. If every vertex has degree two, then $G$ is an odd cycle. If there is a vertex $v$ such that $\deg_G(v) \geq 3$, then let $u_1,u_2,u_3 \in N_G(v)$. The induced subgraph $G[A]$, where $A=\{v,u_1,u_2,u_3\}$, is either a claw or a diamond graph or a complete graph $K_4$ or it is obtained by attaching an edge to a vertex of $K_3$. If $G[A]$ is a claw, then by \cite[Theorem 5.3]{PZ}, $\beta_{3,5}(S/\mathcal{I}_{G[A]}) \neq 0$.  Therefore, by Theorem \ref{betti-induced}, $\beta_{3,5}(S/\mathcal{I}_G) \neq 0$. Since $G$ is a non-bipartite graph, by Theorem \ref{betti-LSS}, $\beta_{3,6}(S/\mathcal{I}_G) \neq 0$, which is a contradiction. Therefore, $G[A]$ is not a claw. Now,  by Proposition \ref{diamond-complete}, $S/\mathcal{I}_{G[A]}$ has not pure resolution. Thus, by Lemma \ref{induced-pure}, $S/\mathcal{I}_G$ has not pure resolution which is a contradiction. Therefore, for every vertex $v$, $\deg_G(v) =2$, and hence, $G$ is an odd cycle.  Let $G_1,\ldots, G_k$ be connected components of $G$. By Lemma \ref{induced-pure}, $S/\mathcal{I}_{G_i}$  has  pure resolution for all $ i$. Therefore, $G_i$ is either a path graph or an odd cycle or a complete bipartite graph. Again by Lemma \ref{induced-pure}, for each pair of $1 \leq i <j \leq k$, $S/\mathcal{I}_{G_i \sqcup G_j}$ has  pure resolution. It follows from \cite[Lemma 2.5]{KM4} and \cite[Theorem 5.3]{PZ} that $G_i$ is either an odd cycle or a path.
	\end{proof}

	\vskip 2mm
	\noindent
	\textbf{Acknowledgements:} The author is grateful to his advisor A. V. Jayanthan for
	his constant support, valuable ideas, and suggestions. The  author thanks the National Board
	for Higher Mathematics, India for the financial support. 
	\bibliographystyle{plain}  %% or 
	\bibliography{Reference}

\begin{thebibliography}{10}

\bibitem{dav}
Davide Bolognini, Antonio Macchia, and Francesco Strazzanti.
\newblock Binomial edge ideals of bipartite graphs.
\newblock {\em European J. Combin.}, 70:1--25, 2018.

\bibitem{AV19}
Aldo Conca and Volkmar Welker.
\newblock Lov\'{a}sz-{S}aks-{S}chrijver ideals and coordinate sections of
  determinantal varieties.
\newblock {\em Algebra Number Theory}, 13(2):455--484, 2019.

\bibitem{HH1}
J\"urgen Herzog, Takayuki Hibi, Freyja Hreinsd\'ottir, Thomas Kahle, and
  Johannes Rauh.
\newblock Binomial edge ideals and conditional independence statements.
\newblock {\em Adv. in Appl. Math.}, 45(3):317--333, 2010.

\bibitem{her3}
J\"urgen Herzog, Antonio Macchia, Sara Saeedi~Madani, and Volkmar Welker.
\newblock On the ideal of orthogonal representations of a graph in
  {$\Bbb{R}^2$}.
\newblock {\em Adv. in Appl. Math.}, 71:146--173, 2015.

\bibitem{THT}
Do~Trong Hoang and Thomas Kahle.
\newblock Hilbert-poincar{\'e} series of parity binomial edge ideals and
  permanental ideals of complete graphs.
\newblock {\em Collect. Math.}, 2020 (to appear).

\bibitem{JACM}
A.~V. Jayanthan and Arvind Kumar.
\newblock Regularity of binomial edge ideals of {C}ohen-{M}acaulay bipartite
  graphs.
\newblock {\em Comm. Algebra}, 47(11):4797--4805, 2019.

\bibitem{JNR}
A.~V. Jayanthan, N.~Narayanan, and B.~V. Raghavendra~Rao.
\newblock Regularity of binomial edge ideals of certain block graphs.
\newblock {\em Proc. Indian Acad. Sci. Math. Sci.}, 129(3):Art. 36, 10, 2019.

\bibitem{JNR1}
A.~V. Jayanthan, N.~Narayanan, and B.~V. Raghavendra~Rao.
\newblock An upper bound for the regularity of binomial edge ideals of trees.
\newblock {\em J. Algebra Appl.}, 18(9):1950170, 7, 2019.

\bibitem{KK19}
Thomas {Kahle} and Jonas {Kr{\"u}semann}.
\newblock {Binomial edge ideals of cographs}.
\newblock {\em arXiv e-prints}, page arXiv:1906.05510, Jun 2019.

\bibitem{KCT}
Thomas Kahle, Camilo Sarmiento, and Tobias Windisch.
\newblock Parity binomial edge ideals.
\newblock {\em J. Algebraic Combin.}, 44(1):99--117, 2016.

\bibitem{KM4}
Dariush Kiani and Sara Saeedi~Madani.
\newblock Binomial edge ideals with pure resolutions.
\newblock {\em Collectanea Mathematica}, 65(3):331--340, Sep 2014.

\bibitem{KM3}
Dariush Kiani and Sara Saeedi~Madani.
\newblock The {C}astelnuovo-{M}umford regularity of binomial edge ideals.
\newblock {\em J. Combin. Theory Ser. A}, 139:80--86, 2016.

\bibitem{AR2}
Arvind Kumar.
\newblock Lov{\'a}sz-saks-schrijver ideals and parity binomial edge ideals of
  graphs.
\newblock {\em European J. Combin.}, 93:103274, 2021.

\bibitem{lss89}
L.~Lov\'{a}sz, M.~Saks, and A.~Schrijver.
\newblock Orthogonal representations and connectivity of graphs.
\newblock {\em Linear Algebra Appl.}, 114/115:439--454, 1989.

\bibitem{lov79}
L\'{a}szl\'{o} Lov\'{a}sz.
\newblock On the {S}hannon capacity of a graph.
\newblock {\em IEEE Trans. Inform. Theory}, 25(1):1--7, 1979.

\bibitem{CG19}
Carla Mascia and Giancarlo Rinaldo.
\newblock Krull dimension and regularity of binomial edge ideals of block
  graphs.
\newblock {\em J. Algebra Appl.}, 19(7):2050133, 17, 2020.

\bibitem{MM}
Kazunori Matsuda and Satoshi Murai.
\newblock Regularity bounds for binomial edge ideals.
\newblock {\em J. Commut. Algebra}, 5(1):141--149, 2013.

\bibitem{FM}
Fatemeh Mohammadi and Leila Sharifan.
\newblock Hilbert function of binomial edge ideals.
\newblock {\em Comm. Algebra}, 42(2):688--703, 2014.

\bibitem{ohh2000}
Hidefumi Ohsugi, J\"urgen Herzog, and Takayuki Hibi.
\newblock Combinatorial pure subrings.
\newblock {\em Osaka J. Math.}, 37(3):745--757, 2000.

\bibitem{oh}
Masahiro Ohtani.
\newblock Graphs and ideals generated by some 2-minors.
\newblock {\em Comm. Algebra}, 39(3):905--917, 2011.

\bibitem{Peeva}
Irena Peeva.
\newblock {\em Graded syzygies}, volume~14 of {\em Algebra and Applications}.
\newblock Springer-Verlag London, Ltd., London, 2011.

\bibitem{KM12}
Sara Saeedi~Madani and Dariush Kiani.
\newblock Binomial edge ideals of graphs.
\newblock {\em Electron. J. Combin.}, 19(2):Paper 44, 6, 2012.

\bibitem{PZ}
Peter Schenzel and Sohail Zafar.
\newblock Algebraic properties of the binomial edge ideal of a complete
  bipartite graph.
\newblock {\em An. \c Stiin\c t. Univ. ``Ovidius'' Constan\c ta Ser. Mat.},
  22(2):217--237, 2014.

\end{thebibliography}
	
\end{document}